\documentclass[12pt,leqno]{article}

\usepackage{graphicx}
\usepackage{float}

\usepackage{amsmath,amsfonts,amssymb,amsthm}

\usepackage{url}
\usepackage[margin=1in]{geometry}

\usepackage{xcolor}

\newtheorem{thm}{Theorem}
\newtheorem{lem}{Lemma}
\newtheorem{prop}[lem]{Proposition}
\newtheorem{cor}[thm]{Corollary}

\theoremstyle{remark}
\newtheorem{remark}[lem]{Remark}

\theoremstyle{definition}

\numberwithin{lem}{section}
\numberwithin{equation}{section}

\newcommand{\NN}{\mathbb{N}}

\newcommand{\ZZ}{\mathbb{Z}}

\newcommand{\bfa}{\mathbf{a}}

\newcommand{\cP}{\mathcal{P}}
\newcommand{\cS}{\mathcal{S}}

\newcommand{\Fa}{F^{(a)}}

\newcommand{\Fab}{F^{(ab)}}
\newcommand{\Gab}{G^{(ab)}}

\begin{document}


\title{On the Continued Fraction Expansion of Almost All
Real Numbers}

\author{Alex Jin,
Shreyas Singh,
Zhuo Zhang,
AJ Hildebrand
\\[1ex]
\url{{ajin7,singh88,zhuoz4,ajh}@illinois.edu}\\[1ex]
Department of Mathematics\\
University of Illinois at Urbana-Champaign\\
Urbana, IL 61801, USA
}

\date{March 14, 2024}


\maketitle

\begin{abstract}
By a classical result of Gauss and Kuzmin, the continued fraction
expansion of a ``random'' real number contains each digit $a\in\NN$
with asymptotic frequency $\log_2(1+1/(a(a+2)))$.

We generalize this result in two directions: First, for certain sets
$A\subset\NN$,  we establish simple explicit formulas for the frequency
with which the continued fraction expansion of a random real number
contains a digit from the set $A$.  For example, we show that digits of
the form $p-1$, where $p$ is prime, appear with frequency $\log_2(\pi^2/6)$. 

Second, we obtain a simple formula for the frequency with which a
string of  $k$ consecutive digits $a$ appears in the continued fraction
expansion of a random real number. In particular, when $a=1$, this
frequency is given by $|\log_2(1+(-1)^k/F_{k+2})|$, where $F_n$
is the $n$th Fibonacci number.

Finally, we compare the frequencies predicted by these results with
actual frequencies found among the first 300 million continued fraction 
digits of $\pi$, and we provide strong statistical evidence that the  
continued fraction expansion of $\pi$ behaves like that of a random real
number.
\end{abstract}

\section{Introduction}
\label{sec:intro}

\subsection{Decimal expansions and normal numbers}

If one picks a ``random'' real number in $(0,1)$ and expands it in base
$10$, then $1/10$ of the digits will be $0$, $1/10$ will be $1$, and so
on. That is, the proportion of digits $0$ among the first $n$ decimal
digits of the number converges to $1/10$ as $n\to\infty$, and the same
holds for any other digit $d\in\{0,1,\dots,9\}$. More generally, any
finite string $d_1\dots d_k$ of $k$ digits in $\{0,1,\dots,9\}$ occurs
in the decimal expansion of the number with frequency $1/10^k$.

A number whose decimal expansion has this property is called
\emph{normal with respect to base $10$}; normality with respect to other
integer bases $b\ge 2$ is defined analogously. 
It is a classical result of Borel (see \cite{borel1909} or \cite[Theorem
8.11]{niven1956}) that almost all real numbers are
normal with respect to all integer bases $b\ge 2$; that is, the set of
numbers that are \emph{not} normal has Lebesgue measure $0$.
Therefore, if we pick a real number ``at random'' (e.g., uniformly from
a finite interval), then, with probability $1$, this number will be
normal, and the statistics of the digits in its decimal (or base $b$)
expansion are well-understood. 

In contrast to such ``almost all'' type results, we know almost nothing
about the statistics of the digits in the expansion of \emph{specific}
irrational numbers. In fact,
with the exception of some specially constructed numbers (for
example, the Champernowne constant \cite{champernowne1933}
$0.1234567891011121314\dots$), for
most ``natural'' irrational constants 
it is not even known whether
each digit occurs infinitely often in the decimal expansion of the
number.  In particular, this is the case for the number
$\pi=3.141592\dots$: Although there exists overwhelming
\emph{statistical} evidence towards the normality of $\pi$, a proof of
even a very weak form of normality seems to be out of reach; see, e.g.,
\cite{aragon2013}, \cite{bailey2014}. 

\subsection{Continued fraction expansions}
In this paper, we consider similar questions with respect to the \emph{continued
fraction expansion} of real numbers, that is, expansions of the form  
\begin{equation}
\label{eq:cf-def}
x=a_0+\cfrac{1}{a_1+\cfrac{1}{a_2+\cfrac{1}{\ddots}}} = [a_0;
a_1,a_2,\dots],
\end{equation}
where $a_0=\lfloor x\rfloor$ and $a_i=a_i(x)$, $i=1,2,\dots$, are positive
integers, which we call the \emph{continued fraction digits}\footnote{%
We do not include the leading term $a_0$ here as $a_0$ may be an
arbitrary integer, while the terms $a_1,a_2,\dots$ are restricted to
positive integers.} of $x$.  It is well-known  (see, e.g., \cite[Theorem
5.11]{niven1956})
that any \emph{irrational} number $x$ has a unique \emph{infinite} continued fraction 
expansion of the form \eqref{eq:cf-def}. 
For example, the continued fraction expansion of $\pi$ is 
\begin{align*}
\pi= 3+\cfrac{1}{7+\cfrac{1}{15+\cfrac{1}{\ddots}}}
= [3; 7, 15, \cdots].
\end{align*}
Conversely, given any integer $a_0$ and any sequence of positive
integers $a_1,a_2,\dots$, there is a unique irrational number $x$ with
continued fraction expansion \eqref{eq:cf-def}.
In this sense, the continued fraction digits are analogous
to the decimal digits of a number, and we can therefore 
ask similar questions as above:
\begin{itemize}

\item \textbf{Statistics of continued fraction digits of a random
real number.}
What can we say about the continued fraction digits of a ``random''
real number (or, equivalently, of almost all real numbers)? 

\item \textbf{Statistics of continued fraction digits of $\pi$.} 
What can we say about the continued fraction digits of \emph{specific}
``natural'' irrational numbers such as the number $\pi$? 
In particular, do the continued fraction digits of $\pi$ 
behave like those of a random real number?
\end{itemize}

Our state of knowledge with respect to these two questions is similar to
that for the usual decimal and base $b$ expansions.  On the one hand,
the statistical nature of the continued fraction expansion of
\emph{almost all} real numbers is now well understood. On the other
hand, we know almost nothing about the statistics of the continued
fraction digits of $\pi$ and most other classical constants.

\subsection{Classical results on the statistics of continued fraction digits}

The study of the statistical nature of continued fraction expansions 
originated with Gauss and was further developed by
Kuzmin \cite{kuzmin1929}, L\'evy \cite{levy1929}, Khinchin
\cite{khinchin-book}, and others.  We mention here two classical
results in this field. 

\emph{Khinchin's Theorem} (see \cite[p.~93]{khinchin-book}, or
\cite[Proposition 4.1.8]{neverending-fractions-book})
states that the geometric mean of the first
$n$ continued fraction digits of a random real number converges, as
$n\to\infty$, to the constant (now known as \emph{Khinchin's constant})
\begin{equation}
\label{eq:khinchin-constant}
K_0=\prod_{k=1}^\infty \left(1+\frac1{k(k+2)}\right)^{\log_2 k}
=2.685452\dots,
\end{equation}
where $\log_2$ denotes the base $2$ logarithm.
That is, almost all real numbers $x$ satisfy
\begin{equation}
\label{eq:khinchin-theorem}
\lim_{n\to\infty}\left(a_1(x)\dots a_n(x)\right)^{1/n}=K_0,
\end{equation}
where $a_i(x)$, $i=1,2,\dots$,
are the continued fraction digits of $x$ defined by \eqref{eq:cf-def}.

The \emph{Gauss-Kuzmin Theorem} (see Lemma \ref{lem:gauss-kuzmin} below)
states that the frequency of a given digit $a\in\NN$
in the continued fraction expansion of a 
random real number is given by
\begin{equation}
\label{eq:gk-distribution-def}
P(a)=\log_2\left(1+\frac1{a(a+2)}\right).
\end{equation} 
That is, almost all real numbers $x$ satisfy
\begin{equation}
\label{eq:Pa-def}
\lim_{n\to\infty}\frac1n\#\{1\le i\le n: a_i(x)=a\}=P(a), \quad a\in\NN.
\end{equation}
Thus, for example, around $\log_2(4/3)\approx 41.5\%$ of
the continued fraction digits of a random real number will be $1$, 
around $\log_2(9/8)\approx 17\%$ of these digits will be $2$, and so on.
The numbers $P(a)$ defined by \eqref{eq:gk-distribution-def}
form a discrete probability distribution on $\NN$, called the 
\emph{Gauss-Kuzmin distribution} and depicted in Figure \ref{fig:gk}.
\begin{figure}[H]
\begin{center}
\includegraphics[width=0.7\textwidth]{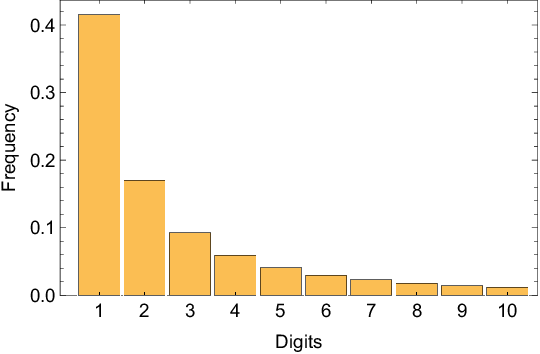}
\caption{The Gauss-Kuzmin distribution 
$P(a)=\log_2(1+\frac1{a(a+2)})$, $a=1,2,\dots$. ($\log_2 t =\log
t/\log 2$ denotes the base $2$ logarithm.)}
\label{fig:gk}
\end{center}
\end{figure}

As natural extensions of the Gauss-Kuzmin result on the frequency of
a single continued fraction digit $a$, one can consider the
following frequencies\footnote{For simplicity of notation,
we use the same symbol, $P(\dots)$,
to denote single digit frequencies, subset frequencies,
and frequencies of strings. This should not cause any problems as the
meaning will always be clear from the context.}:
\begin{itemize}
\item\textbf{Subset frequencies:}  
The frequency, $P(A)$, with which the continued fraction digits of
a random real number belong to a given subset $A\subset \NN$.
More precisely, given a set $A\subset\NN$ and 
an irrational number $x$ with continued fraction
digits $a_i(x)$, $i=1,2,\dots$, 
we are interested in the limit 
\begin{equation}
\label{eq:PA-def}
\lim_{n\to\infty}\frac1n\#\{1\le i\le n: a_i(x)\in A\}.
\end{equation}

\item 
\textbf{String frequencies:}
The frequency, $P(\bfa)$, with which the continued fraction expansion of 
a random real number contains a given finite string 
$\bfa=(a_1,\dots,a_k)$, of digits $a_i\in\NN$. More precisely, given an
irrational number $x$ with continued fraction digits $a_i(x)$, 
we are interested in the limit 
\begin{equation}
\label{eq:Pbfa-def}
\lim_{n\to\infty}\frac1n\#\{1\le i\le n: 
a_{i+1}(x)=a_1,\dots,a_{i+k}(x)=a_k\}.
\end{equation}

\end{itemize}

It follows from the Gauss-Kuzmin theory (see Lemmas
\ref{lem:gauss-kuzmin} and \ref{lem:gauss-kuzmin-generalized}) that, for
almost all real numbers $x$, the limits in \eqref{eq:PA-def} and
\eqref{eq:Pbfa-def} exist and depend only on the set $A$ (respectively
the string $\bfa$). This justifies the use of the notation $P(A)$ and
$P(\bfa)$ for these frequencies. The Gauss-Kuzmin theory also yields  
explicit formulas for the frequencies $P(A)$ and $P(\bfa)$, but these
formulas are in general quite complicated,  involving either infinite
sums (in the case of $P(A)$ for infinite sets $A$) or continued fraction
convergents (in the case of $P(\bfa)$).  In our main results, Theorems
\ref{thm:Pk}--\ref{thm:strings} below, we exhibit a class of
sets $A$ and strings $\bfa$ for which these frequencies have
surprisingly simple closed-form expressions.

\subsection{Statement of main results}

We first consider the frequencies
$P(A)$.  In the case when the set $A$ is an arithmetic progression, Nolte
\cite{nolte1990} (see also Girstmair \cite{girstmair2020}
and \cite[Proposition 4.1.5]{iosifescu-kraaikamp-book})
established a closed form expression for the subset frequency $P(A)$ in
terms of the Euler gamma function.

In our first two results, we establish similarly
simple expressions for $P(A)$ in the case when $A$ is a set of
shifted prime powers or shifted squares, defined as follows:
\begin{align}
\label{eq:Pk-def}
\cP_k&=\{p^k-1\in\NN: p\text{ prime }\}\quad (k\in\NN), 
\\
\label{eq:S-def}
\cS&=\{n^2-1: n\in\NN, n>1\}.
\end{align}

\begin{thm}[Shifted prime powers]
\label{thm:Pk}
Let $k$ be a positive integer. Then the frequency of digits of the form 
$p^k-1$, where $p$ is prime, in the continued fraction expansion of 
a random real number is given by 
\begin{equation}
\label{eq:Pk-result}
P(\cP_k)=\log_2\zeta(2k),
\end{equation}
where $\zeta(s)$ is the Riemann zeta function. In particular,
the frequency of digits of the form 
$p-1$, where $p$ is prime, in the continued fraction expansion of 
a random real number is given by 
\begin{equation}
\label{eq:P1-result}
P(\cP_1)=\log_2\zeta(2)=\log_2\frac{\pi^2}{6}=0.718029\dots.
\end{equation}
\end{thm}

\begin{thm}[Shifted squares]
\label{thm:S}
The frequency of digits of the form 
$n^2-1$, where $n\in\NN$, in the continued fraction expansion of 
a random real number is given by 
\begin{equation}
\label{eq:PS-result}
P(\cS)=\log_2\frac{8\pi}{e^\pi-e^{-\pi}}=0.121832\dots 
\end{equation}
\end{thm}

The relatively large value $0.71\dots$ of the frequency
\eqref{eq:P1-result} for ``shifted prime'' digits 
is due to the fact that the set $\cP_1=\{1,2,4,6,10,12,16,18,22,\dots\}$ of
shifted primes is rather dense at the beginning. In particular, it 
contains the digits $1$ and $2$, the two most
frequent digits in the continued fraction expansion of a random number,
as well as two of the next four most frequent digits, namely $4$ and $6$. 
The combined
frequencies of the digits $1,2,4,6$ in $\cP_1$ (i.e., the sum of the
Gauss-Kuzmin probabilities $P(a)$ for $a=1,2,4,6$) alone is around $0.67\dots$ 
and thus accounts for the bulk of the shifted prime digit frequency given
in \eqref{eq:P1-result}.  By contrast, 
the set $\cS$ of shifted squares contains neither of the digits $1$ and
$2$ and only one of the first seven digits; hence the much
smaller shifted square digit frequency in \eqref{eq:PS-result}.

We next consider occurrences of  
finite strings $\bfa=(a_1,\dots,a_k)$ of continued fraction digits.
A key difference between decimal expansions and continued fraction
expansions of a random real number is that, while the consecutive decimal
digits behave like \emph{independent} random variables, 
this is not the case for the continued fraction digits.   
That is, the string frequency
$P((a_1,\dots,a_k))$ is, in general, \emph{not} equal to 
the product of the corresponding single digit frequencies, $P(a_1)\dots
P(a_k)$.  In particular, the frequency
$P((a,\dots,a))$ of a ``run'' of $k$ consecutive digits $a$
is \emph{not} equal to $P(a)^k$, the $k$-th power of the
single digit frequency.
In the following theorem, we determine these run frequencies explicitly
in terms of a simple two-term recurrence sequence, which reduces to the
Fibonacci sequence when $a=1$.

\begin{thm}[Strings of identical digits]
\label{thm:strings}
Let $a$ and $k$ be positive integers. Then
the frequency of a string of $k$ consecutive digits $a$ in the continued
fraction expansion of 
a random real number is given by 
\begin{equation}
\label{eq:string-result}
P(\underbrace{(a,\dots,a)}_k)=
\left|\log_2\left(1+\frac{(-1)^k} {\left(\Fa_{k+2}\right)^2}\right)\right|,
\end{equation}
where $\Fa_n$, $n=1,2,\dots$ is defined by 
\begin{equation}
\label{eq:tk-def}
\Fa_1= \Fa_2=1,\quad
\Fa_n=a\Fa_{n-1} +\Fa_{n-2}
\quad(n\ge 3).
\end{equation}
In particular, 
the frequency of a string of $k$ consecutive digits $1$ in the continued
fraction expansion of 
a random real number is given by 
\begin{equation}
\label{eq:string1-result}
P(\underbrace{(1,\dots,1)}_k)=
\left|\log_2\left(1+\frac{(-1)^k} {F_{k+2}^2}
\right)\right|,
\end{equation}
where  $F_n$ is the $n$th Fibonacci number.
\end{thm}

By estimating the term $F_{k+2}$ in \eqref{eq:string1-result}   
via Binet's formula (see, e.g., \cite[(10.14.1)]{hardy-wright2008}),
we obtain the following corollary.

\begin{cor}
\label{cor:strings}
As $k\to\infty$, the frequency \eqref{eq:string1-result} of a string of
$k$ consecutive digits $1$ in the continued fraction expansion of a
random real number satisfies
\begin{equation}
\label{eq:string1-result2}
P(\underbrace{(1,\dots,1)}_k)
=\frac{5}{\Phi^{2(k+2)}\log 2}\left(1 +
O\left(\frac1{\Phi^{2(k+2)}}\right)\right),
\end{equation}
where $\Phi=(1+\sqrt{5})/2$ is the Golden Ratio.
\end{cor}

The approximations provided by the corollary are remarkably accurate,
even for very small values of $k$.  Table
\ref{table:string1frequencies} below shows the exact string frequencies
given by Theorem \ref{thm:strings}, i.e.,
$|\log_2(1+(-1)^k/F_{k+2}^2)|$, the approximations provided by Corollary
\ref{cor:strings}, i.e., $5/ (\Phi^{2(k+2)}\log 2)$, and the relative error
in these approximations.

\begin{table}[H]
\begin{center}
\renewcommand{\arraystretch}{1.5}
\begin{tabular}{|c|l|l|l|}
\hline
$k$ & Exact value & Approximation & Relative error 
\\
\hline
 1 & $\log_2 (4/3)=0.415037\dots$ &
   0.401993\dots& 3.24\%\\
 2 & $\log_2 (10/9)=0.152003\dots$ &
   0.153548\dots & 1.00\%\\
 3 & $\log_2 (25/24)=0.058893\dots$
   & 0.058650\dots & 0.41\% \\
 4 & $\log_2 (65/64)= 0.022367\dots$
   & 0.022402\dots & 0.15\% \\
 5 & $\log_2 (169/168)= 0.008562\dots$ &
 0.008556\dots & 0.06\% 
\\
\hline
\end{tabular}
\caption{Exact and approximate values for the frequencies of a string of
$k$ digits $1$ in the continued fraction expansion of a random real
number.} 
\label{table:string1frequencies}.
\end{center}
\end{table}

\subsection{Outline of paper}

The remainder of this paper is organized as follows. In Section
\ref{sec:background} we review the elementary theory of continued
fractions, and we state the key results from the metric theory of
continued fractions that we will use in proving our results.
Section \ref{sec:subsets} contains the proofs of Theorems
\ref{thm:Pk} and \ref{thm:S}, and Section \ref{sec:strings} contains the
proofs of Theorem \ref{thm:strings} and Corollary \ref{cor:strings}. 
In Section \ref{sec:pi} we present empirical data on the frequencies
predicted by these theorems based on the first 300 million continued
fraction digits of $\pi$, and the results of statistical tests comparing
the actual and predicted frequencies.  We conclude in Section
\ref{sec:conclusion} with some remarks on possible extensions 
of our results and open problems.


\section{Background on continued fractions}
\label{sec:background}

We begin by recalling some key definitions and results from the elementary
theory of continued fractions.  Details and proofs can be found in
\cite[Chapter 2]{neverending-fractions-book}, \cite[Chapter
9]{hardy-wright2008}, \cite[Chapters I--II]{khinchin-book}, and
\cite[Chapter 5]{niven1956},

A \emph{continued fraction} is a finite or infinite expression of the form
\begin{equation}
\label{eq:cf-def1}
a_0+\cfrac{1}{a_1+\cfrac{1}{a_2+\cfrac{1}{\ddots}}} = [a_0;
a_1,a_2,\dots],
\end{equation}
where $a_0$ is an arbitrary integer, and $a_1,a_2,\dots $ are positive
integers.  

Clearly, any \emph{finite} (i.e., terminating) continued fraction
represents a rational number. Conversely, any rational number can be
represented as a finite continued fraction $[a_0;a_1,\dots,a_n]$,
and this representation is unique if we require $a_n>1$. 

An \emph{infinite} continued fraction is defined as the limit, as
$n\to\infty$,  of the finite continued fractions obtained by truncating
the given infinite continued fraction after $n$ terms:
\begin{equation}
\label{eq:cf-infinite-def}
[a_0;a_1,a_2,\dots]=
\lim_{n\to\infty} [a_0;a_1,a_2,\dots,a_n].
\end{equation}
It is known (see, e.g., \cite[Theorem 5.11]{niven1956})
that any \emph{irrational} real number has a unique infinite
continued fraction expansion; that is, there exists a unique sequence
$a_0,a_1,a_2,\dots$ satisfying $a_0\in\ZZ$ and $a_i\in\NN$ for
$i=1,2,\dots$, such that 
\begin{equation}
\label{eq:cf-infinite-representation}
x= \lim_{n\to\infty} [a_0;a_1,a_2,\dots,a_n].
\end{equation}
Conversely, given any integer $a_0$ and any sequence
$a_1,a_2,\dots$ of positive integers, the limit
\eqref{eq:cf-infinite-def} exists and represents an irrational number.

The numbers $a_i$ in the representation
\eqref{eq:cf-infinite-representation}
can be computed recursively by
\begin{align}
\label{eq:cf-algorithm1}
a_0&=\lfloor x\rfloor, \quad r_0 = x-a_0,
\\
\label{eq:cf-algorithm2}
a_n&=\left\lfloor\frac1{r_{n-1}}\right\rfloor,\quad r_n=\frac1{r_{n-1}}-a_n
\quad (n\ge 1).
\end{align}

Given an irrational number $x$ with continued fraction representation 
\eqref{eq:cf-infinite-representation}, the rational numbers represented
by the truncated continued fractions $[a_0;a_1,\dots,a_n]$,
$n=0,1,2,\dots$,  are called the \emph{convergents} of $x$. The $n$th
convergent is traditionally denoted by $p_n/q_n$; that is, we have
\begin{equation}
\label{eq:convergents-def}
[a_0;a_1,\dots,a_n]=\frac{p_n}{q_n},
\end{equation}
with the convention
that the denominator $q_n$ is a positive integer and the
numerator $p_n$ is an integer relatively prime to $q_n$.
For example, the first six convergents for $\pi$ (corresponding to
indices $n=0,1,\dots,5$) are  
\begin{equation}
\label{eq:pi-convergents}
\frac{3}{1},\quad \frac{22}{7},\quad\frac{333}{106},\quad
\frac{355}{113},\quad \frac{103993}{33102},\quad
\frac{104348}{33215}.
\end{equation}

A key property of a continued fraction convergent $p_n/q_n$ for an
\emph{irrational} number $x$ is that it represents the best rational
approximation to $x$ among all rational numbers with denominator bounded
by $q_n$. In other words, if $p/q$ is a rational number with denominator
satisfying $|q|\le q_n$, then either $p/q=p_n/q_n$ or
$|x-p/q|>|x-p_n/q_n|$.

The numerators $p_n$ and denominators $q_n$ of the convergents 
satisfy the recurrences
\begin{align}
\label{eq:pn-recurrence}
&p_{-1}=1,\quad p_0=a_0,\quad p_n=a_np_{n-1}+p_{n-2}\quad (n\ge 1),
\\
\label{eq:qn-recurrence}
&q_{-1}=0,\quad q_0=1,\quad q_n=a_nq_{n-1}+q_{n-2}\quad (n\ge 1).
\end{align}
From these recurrences it can be proved by induction that  
\begin{equation}
\label{eq:pn-qn-diff}
\frac{p_n}{q_n}-\frac{p_{n-1}}{q_{n-1}}=\frac{(-1)^{n-1}}{q_nq_{n-1}}.
\end{equation}
Since, by \eqref{eq:qn-recurrence}, the numbers $q_n$, $n\ge 1$, form a
strictly
increasing sequence of positive integers, \eqref{eq:pn-qn-diff}
implies that the sequence of convergents $p_n/q_n$, $n\ge 1$, converges
to a finite limit, given by the real number $x$ represented by the
infinite continued fraction $[a_0;a_1,a_2,\dots]$. It also follows from
\eqref{eq:pn-qn-diff} that the even-indexed convergents form an increasing
sequence and the odd-indexed convergents form a decreasing
sequence; that is, we have
\begin{equation}
\label{eq:convergents-sequence}
\frac{p_0}{q_0}<\frac{p_2}{q_2}<\frac{p_4}{q_4}\cdots < x < \cdots
\frac{p_5}{q_5}<\frac{p_3}{q_3}<\frac{p_1}{q_1}.
\end{equation}


We conclude this section by stating two key results from the metric
theory of continued fractions that we will need in proving Theorems
\ref{thm:Pk}, \ref{thm:S}, and \ref{thm:strings}. 
These results lie much deeper than the elementary properties of
continued fractions cited above, and their proofs are quite involved, 
requiring either very delicate elementary estimates (see, for example,
\cite{khinchin-book}), or the use of
results from ergodic theory (see, for example,
\cite{dajani-kraaikamp-book} or \cite{iosifescu-kraaikamp-book}). 
Both results are special cases of a more general theorem 
which in essence states that the transformation 
$T(x)=1/x-\lfloor1/x\rfloor$ that 
generates continued fractions via the recurrence
\eqref{eq:cf-algorithm2} 
is an \emph{ergodic} transformation on the interval $[0,1]$
with invariant measure $\mu([0,x])=\log_2(1+x)$.

\begin{lem}
[Gauss-Kuzmin Theorem, {\cite[Proposition 4.1.1]{iosifescu-kraaikamp-book}}]
\label{lem:gauss-kuzmin}
For almost all real numbers $x$, each digit $a$, 
$a=1,2,\dots$, appears in the
continued fraction expansion of $x$ with frequency 
\begin{equation}
\label{eq:gk-distribution}
P(a)=\log_2\left(1+\frac1{a(a+2)}\right).
\end{equation}
\end{lem}


\begin{lem}
[Generalized Gauss-Kuzmin Theorem,
{\cite[Proposition 4.1.2]{iosifescu-kraaikamp-book}}]
\label{lem:gauss-kuzmin-generalized}
Let $k$ be a positive integer and let 
$\bfa=(a_1,\dots,a_k)$ be a finite string of positive integers.
Then, for almost all real numbers $x$, the string $\bfa$ appears in the
continued fraction expansion of $x$ with frequency 
\begin{equation}
\label{eq:gk-generalized-distribution}
P(\bfa)=
\left|\log_2\left(1+\frac{(-1)^k}{(p_k+q_k)(q_k+q_{k-1})}\right)\right|
\end{equation}
where $p_i$ and $q_i$ are the numerators and
denominators of the convergents $p_i/q_i=[0;a_1,\dots,a_i]$, $i\le k$, 
defined recursively by \eqref{eq:pn-recurrence}
and \eqref{eq:qn-recurrence} with respect to the \emph{finite}
continued fraction 
$[0;a_1,\dots,a_k]$.
\end{lem}


\section{Proof of Theorems \ref{thm:Pk} and \ref{thm:S}}
\label{sec:subsets}

The proofs hinge on the following proposition, which expresses the
frequency $P(A)$ in terms of an infinite product over the set $A$,
along with evaluations for these products in the case of sets of
the form $\cP_k$ and $\cS$.

\begin{prop}[Explicit formula for subset frequencies]
\label{prop:subsets}
Let $A\subset \NN$.  Then, for almost all real numbers $x$, the 
continued fraction expansion of $x$ contains a digit from the set $A$ with
frequency 
\begin{equation}
\label{eq:PA-def2}
P(A)=\log_2\prod_{a\in A}\left(1-\frac1{(a+1)^2}\right)^{-1}=-\log_2\prod_{a\in A}\left(1-\frac1{(a+1)^2}\right).
\end{equation}
That is, almost all real numbers $x$ satisfy
\begin{equation}
\label{eq:PA-def3}
\lim_{n\to\infty}\frac1n\#\{1\le i\le n: a_i(x)\in A\} =P(A),
\end{equation}
with $P(A)$ given by \eqref{eq:PA-def2}.
\end{prop}

\begin{proof}
By Lemma \ref{lem:gauss-kuzmin} the frequency $P(A)$ defined in
\eqref{eq:PA-def3} exists for almost all real numbers and is given by
\begin{equation*}
\label{eq:PA-formula1}
P(A)=\sum_{a\in A} P(a) =\sum_{a\in A}\log_2 \left(1+\frac{1}{a(a+2)}\right)
=\log_2\prod_{a\in A}\left(1+\frac{1}{a(a+2)}\right).
\end{equation*}
Since
\begin{equation*}
\label{eq:Pa-algebra}
1+\frac{1}{a(a+2)}= \frac{(a+1)^2}{(a+1)^2-1}
=\left(1-\frac1{(a+1)^2}\right)^{-1},
\end{equation*}
this yields \eqref{eq:PA-def2}, as desired.
\end{proof}

\begin{proof}[Proof of Theorem \ref{thm:Pk}]
Suppose $A$ is a shifted power set,
i.e., $A=\cP_k=\{p^k-1: p\text{ prime }\}$, where $k\in\NN$. 
In this case the product on the
right of \eqref{eq:PA-def2} becomes
\begin{align}
\label{eq:thm1-proof}
\prod_{a\in \cP_k}\left(1-\frac1{(a+1)^2}\right)^{-1}
&=\prod_{p\text{ prime }}\left(1-\frac1{p^{2k}}\right)^{-1}=\zeta(2k),
\end{align}
where we have used the Euler product formula for the Riemann zeta
function
$\zeta(s)=\sum_{n=1}^\infty 1/n^s$ 
(see, e.g., \cite[Theorem 11.7]{apostol-analytic-number-theory-book})
\begin{equation*}
\label{eq:euler-product}
\zeta(s)=\prod_{p\text{ prime }}\left(1-\frac1{p^s}\right)^{-1}
\quad (\operatorname{Re}(s)>1).
\end{equation*}
Substituting \eqref{eq:thm1-proof} into \eqref{eq:PA-def2} yields
the desired formula \eqref{eq:Pk-result} for $P(\cP_k)$.

The formula \eqref{eq:P1-result} for $P(\cP_1)$ follows on noting that
$\zeta(2)=\sum_{n=1}^\infty 1/n^2=\pi^2/6$.
\end{proof}

\begin{proof}[Proof of Theorem \ref{thm:S}]
When $A$ is the set $\cS=\{n^2-1: n\in\NN, n>1\}$ of shifted squares,
we have
\begin{align}
\label{eq:thm2-proof1}
\prod_{a\in \cS}\left(1-\frac1{(a+1)^2}\right)^{-1}
&=\prod_{n=2}^\infty\left(1-\frac1{n^4}\right)^{-1}
\\
\notag
&=\left(\prod_{n=2}^\infty\left(1-\frac1{n^2}\right)
\prod_{n=2}^\infty\left(1+\frac1{n^2}\right)\right)^{-1}.
\end{align}
The first of the two products on the right is easy to evaluate by a
telescoping argument:
\begin{align}
\label{eq:thm2-proof-prod1}
\prod_{n=2}^\infty \left(1-\frac1{n^2}\right)
=\lim_{N\to\infty} \prod_{n=2}^N \frac{n-1}{n}\cdot \frac{n+1}{n}
=\lim_{N\to\infty} \frac{N+1}{2N}=\frac12.
\end{align}
To evaluate the second product, we use Euler's product formula for the
sine function 
(see \cite{eberlein1977} or 
\cite[Section 3.23]{titchmarsh-theory-of-functions-book})
\begin{equation}
\label{eq:euler-sinc-formula}
\frac{\sin(\pi z)}{\pi z}=\prod_{n=1}^\infty
\left(1-\frac{z^2}{n^2}\right),
\end{equation}
which is valid for all complex numbers $z$. 
Setting $z=i$ in \eqref{eq:euler-sinc-formula},
it follows that
\begin{equation}
\label{eq:euler-sinc-formula1}
\frac{\sin(\pi i)}{\pi i}=
\prod_{n=1}^\infty \left(1+\frac{1}{n^2}\right)
=2\prod_{n=2}^\infty \left(1+\frac{1}{n^2}\right).
\end{equation}
Hence
\begin{align}
\label{eq:thm2-proof-prod2}
\prod_{n=2}^\infty \left(1+\frac1{n^2}\right)
&=\frac{\sin (\pi i)}{2\pi i}=\frac{e^{\pi}-e^{-\pi}}{4\pi}.
\end{align}
Substituting
\eqref{eq:thm2-proof-prod1} and \eqref{eq:thm2-proof-prod2}
into \eqref{eq:thm2-proof1} and applying Proposition 
\ref{prop:subsets} yields
\begin{equation}
\label{eq:thm2-proof3}
P(\cS)= 
\log_2\left(\prod_{n=2}^\infty \left(1-\frac1{n^2}\right)
\prod_{n=2}^\infty \left(1+\frac1{n^2}\right)\right)^{-1}
=\log_2 \left(\frac{8\pi}{e^{\pi}-e^{-\pi}}\right),
\end{equation}
as desired.
\end{proof}


\section{Proof of Theorem \ref{thm:strings} and Corollary
\ref{cor:strings}}
\label{sec:strings}

\begin{proof}[Proof of Theorem \ref{thm:strings}]
Let $a$ and $k$ be positive integers, and let $\bfa=(a,\dots,a)$ denote
the string consisting of $k$ consecutive digits $a$.  By Lemma
\ref{lem:gauss-kuzmin-generalized}, the frequency $P(\bfa)$ is given by
\begin{equation}
\label{eq:gk-generalized-distribution0}
P(\bfa)=
\left|\log_2\left(1+\frac{(-1)^k}{(p_k+q_k)(q_k+q_{k-1})}\right)\right|,
\end{equation}
where $p_k$ and $q_k$ are defined recursively by 
\begin{align}
\label{eq:pn-recurrence0}
&p_{-1}=1,\quad p_0=0,\quad p_n=ap_{n-1}+p_{n-2}\quad (n=1,\dots,k),
\\
\label{eq:qn-recurrence0}
&q_{-1}=0,\quad q_0=1,\quad q_n=aq_{n-1}+q_{n-2}\quad (n=1,\dots,k).
\end{align}
Setting  
\begin{equation}
\label{eq:rk-sk-def}
r_n=q_n+q_{n-1},\quad s_n=p_n+q_{n},
\end{equation}
it follows from \eqref{eq:pn-recurrence0} and \eqref{eq:qn-recurrence0}
that 
\begin{align}
\label{eq:rn-recurrence}
&r_0=1,\quad r_1=a+1,\quad r_n=ar_{n-1}+r_{n-2}\quad (2\le n\le k),
\\
\label{eq:sn-recurrence}
&s_0=1,\quad s_1=a+1,\quad s_n=as_{n-1}+s_{n-2}\quad (2\le n\le k).
\end{align}
Thus, the sequences $r_n$ and $s_n$ satisfy the same recurrence as
the sequence $\Fa_n$ of Theorem \ref{thm:strings}. In addition, their
values at $n=0$ and $n=1$, i.e., $r_0=s_0=1$ and $r_1=s_1=a+1$,
agree with the values of $\Fa_n$ at $n=2$ and $n=3$, namely
$\Fa_2=1$ and $\Fa_3=a\Fa_2+\Fa_1=a+1$. Therefore we have
\begin{equation*}
\label{eq:rn-formula}
r_n=s_n=\Fa_{n+2}\quad (1\le n\le k).
\end{equation*}
Hence
\begin{align*}
\log_2\left(1+\frac{(-1)^k}{(p_k+q_k)(q_k+q_{k-1})}\right)
&=\log_2\left(1+\frac{(-1)^k}{r_ks_k}\right)
=\log_2\left(1+\frac{(-1)^k}{\left(\Fa_{k+2}\right)^2}\right).
\end{align*}
Substituting this into 
\eqref{eq:gk-generalized-distribution0} yields the desired result. 
\end{proof}

\begin{proof}[Proof of Corollary \ref{cor:strings}]
Let 
$\Phi=(1+\sqrt{5})/2=1.618034\dots$ denote the Golden Ratio and let
$\overline{\Phi}=(1-\sqrt{5})/2=1/\Phi=-0.618034\dots$ denote its conjugate.
Then Binet's formula for Fibonacci numbers
(see, e.g., \cite[(10.14.1)]{hardy-wright2008}) yields  
\begin{equation}
\label{eq:binet}
F_n=\frac{\Phi^n-{\overline{\Phi}}^n}{\sqrt{5}}=
\frac{\Phi^n}{\sqrt{5}}+O\left(\frac1{\Phi^n}\right),
\end{equation}
and hence 
\begin{equation}
\label{eq:binet2}
\frac1{F_{k+2}^2}
=\frac{5}{\Phi^{2(k+2)}+O(1)}
=\frac{5}{\Phi^{2(k+2)}}\left(1+O\left(\frac{1}{\Phi^{2(k+2)}}\right)\right).
\end{equation}
Combining this with the elementary estimate 
\begin{equation}
\label{eq:log-estimate}
\log_2(1+x)=\frac{\log(1+x)}{\log 2}
= \frac{x}{\log 2}\left(1 +O(x)\right),
\end{equation}
which holds as $x\to0$, 
it follows from Theorem \ref{thm:strings} that 
\begin{equation}
\label{eq:string1-result2a}
P(\underbrace{(1,\dots,1)}_k)=
\left|\log_2\left(1+\frac{(-1)^k} {F_{k+2}^2}\right)\right|
=\frac{5}{\Phi^{2(k+2)}\log 2}\left(1 +
O\left(\frac1{\Phi^{2(k+2)}}\right)\right),
\end{equation}
as claimed.
\end{proof}

\begin{remark}
By replacing the $O$-estimates in \eqref{eq:binet}, \eqref{eq:binet2},
and \eqref{eq:log-estimate} with explicit inequalities, one can obtain
an explicit value for the $O$-constant in 
\eqref{eq:string1-result2a} 
(i.e., the estimate 
\eqref{eq:string1-result2}
of Corollary \ref{cor:strings}). 
In particular, one can show that an admissible value for this $O$-constant
is $1$; that is, the expression
represented by the $O$-term in \eqref{eq:string1-result2a} 
is bounded in absolute value 
by $\le 1/\Phi^{2(k+2)}$ for all $k$. 
We leave out the details, which are routine, though rather tedious.
\end{remark}

\section{Statistical analysis of continued fraction digits of $\pi$}
\label{sec:pi}

The number $\pi$ is arguably the most famous mathematical constant, and
it has been the subject of more studies---both theoretical and
numerical---than any other mathematical constant.  The decimal
expansion of $\pi$ has been calculated to more than twenty trillion
digits \cite{trueb2016} and has been subjected to extensive statistical
analyses; see, for example, \cite{bailey1988}, \cite{bailey2014}.

By comparison, the continued fraction digits of $\pi$ have received much
less attention in the literature, with most of the empirical studies
focusing on \emph{single digit} frequencies among the digits of $\pi$.
For example, \cite[Figure 3.2]{neverending-fractions-book} shows the
deviations between observed and predicted frequencies of digits $1$
through $15$ based on the first 100 million continued fraction digits of $\pi$.

In light of Theorems \ref{thm:Pk}--\ref{thm:strings} it seems natural to
compare, in a statistically rigorous manner, the frequencies predicted
by these results (i.e., the actual frequencies in the
continued fraction expansion of a ``random'' real number), with the
corresponding empirical frequencies based on the continued fraction
digits of $\pi$. 

\subsection{Methodology}

Our analysis is based on a data set of $300$ million continued
fraction digits of $\pi$.
We first split the full data set of $300$ million digits into  $B$
disjoint blocks of length $N$ each (so that $B\cdot N=3\cdot 10^8$)
and apply a given statistical test to
each of these $B$ blocks.  The test outputs a test statistic (for
example, a z-score) for each of the $B$ blocks.  We then employ a
second statistical test to compare the $B$ values of the test statistic
to their predicted distribution.  Such a two-stage testing procedure is
based on NIST recommendations for randomness tests, and it
results in improved reliability when compared to applying a test to the
full data set; see \cite{pareschi2007}. 

To ensure robustness of the results, we carry out our tests with three
different values for the block size, namely blocks of $N=$ 250,000; 500,000; and 1,000,000 digits.
This corresponds to splitting the full set of $300$ million digits into
$B=$ 12,000; 6,000; and 3,000 blocks, respectively.

We carry out this process with the following tests.

\paragraph{Test I: Z-score test for shifted primes.}
By Theorem \ref{thm:Pk}, the predicted frequency of
continued fraction digits of the form $p-1$, where $p$ is prime, is 
\begin{equation}
\label{eq:q-value-p1}
q=\log_2 (\pi^2/6)=0.718029\dots 
\end{equation}
To test this prediction, we 
model the continued fraction digits as Bernoulli random variables with
parameter $q$ and compute, for each block of size $N$,
the corresponding $z$-score, given by
\begin{equation}
\label{eq:z-score-def}
z=\frac{N_1-Nq}{\sqrt{Nq(1-q)}},
\end{equation}
where $N_1$ is the actual number of
digits of the form $p-1$ among the $N$ digits in the block.

We expect these $z$-scores to be approximately normally distributed.  To
test this, we apply the Kolmogorov-Smirnov and Anderson-Darling
goodness-of-fit tests (see \cite[1.3.5.14 and
1.3.5.16]{nist2003-chapter1}) to this set of $z$-scores.

\paragraph{Test II: Z-score test for shifted squares.}
We employ an analogous testing procedure to digits of the form $n^2-1$,
where $n\in\NN$, $n>1$, with the predicted probability  being given by Theorem
\ref{thm:S}, namely
\begin{equation}
\label{eq:q-value-s}
q=\log_2\left(\frac{8\pi}{e^\pi-e^{-\pi}}\right)=0.121832\dots
\end{equation}

\paragraph{Test III: Longest run of $1$s.}
Our third test is based on the predicted frequencies of ``runs'' of
digits $1$ given by Theorem \ref{thm:strings}, and the approximations
provided by Corollary \ref{cor:strings}.  Specifically, we determine, for
each block,
the length of the longest run of consecutive digits $1$, and
then compute the average of these longest run lengths across all
blocks.  This average can be considered as an approximation to the
\emph{expected} length of the longest string of $1$s in an appropriate
random model.  

While consecutive continued fraction digits are not independent, the
result of Corollary \ref{cor:strings} and the numerical data provided in Table
\ref{table:string1frequencies} suggest that, for ``long'' strings of
$1$, each additional digit $1$ decreases the frequency by a factor
$q=1/\Phi^2\approx 0.381966\dots$. This suggests that, in the
context of ``long'' runs of $1$s, digits behave approximately like
Bernoulli random variables with probability $q=1/\Phi^2$. Let 
$L_{N,q}$ denote the length of the longest run of $1$s in a sequence of
$N$ independent Bernoulli random variables that take on the value $1$
with probability $q$, and $0$ with probability $1-q$. It is known 
(see \cite[(5)]{schilling1990}) that, as $N\to\infty$, we have
\begin{equation}
\label{eq:longest-run-prediction}
E(L_{N,q})=\frac{\log(N(1-q))}{\log (1/q)} + \gamma
\log(1/q)-\frac12 + o(1),
\end{equation}
where $\gamma$ is the Euler-Mascheroni constant. Ignoring the term $o(1)$,    we
compare the prediction \eqref{eq:longest-run-prediction} with the
average length of the longest run computed from our data. 

\paragraph{Implementation notes.}
We used \emph{Mathematica 13.1} to generate the set of $300$ million
continued fraction digits of $\pi$ that formed the basis of our analysis.
We used Python for the first stage tests described above and
\emph{Mathematica} for subsequent statistical analysis and visualization,
such as generating histograms of z-scores and applying the
Kolmogorov-Smirnov and Anderson-Darling goodness-of-fit tests to the
z-scores produced by Tests I and II.

\subsection{Results}

\paragraph{Tests I and II:}
Figure \ref{fig:zscores} shows the results of the z-score tests, applied 
to frequencies of digits of the form $p-1$ (top row)
and $n^2-1$ (bottom row) for block sizes $N=$ 250,000; 500,000; and 1,000,000.  The actual distributions of $z$-scores are shown as
density histograms, with the density function of a standard normal
distribution overlayed.  The z-scores were computed via formula
\eqref{eq:z-score-def} with the probabilities $q$ being those given by
\eqref{eq:q-value-p1} for the case of shifted primes and 
\eqref{eq:q-value-s} for the case of shifted squares.

\begin{figure}[H]
\begin{center}
\includegraphics[width=.3\textwidth]{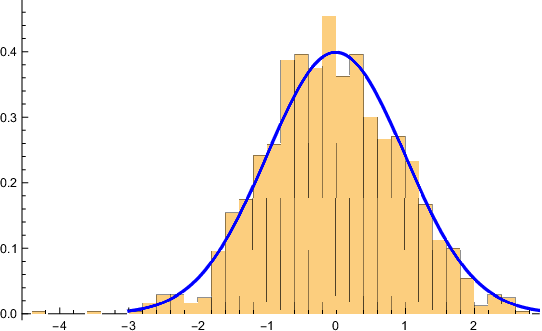}
\hspace{1em}
\includegraphics[width=.3\textwidth]{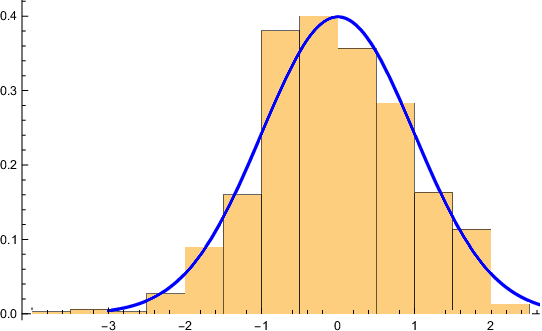}
\hspace{1em}
\includegraphics[width=.3\textwidth]{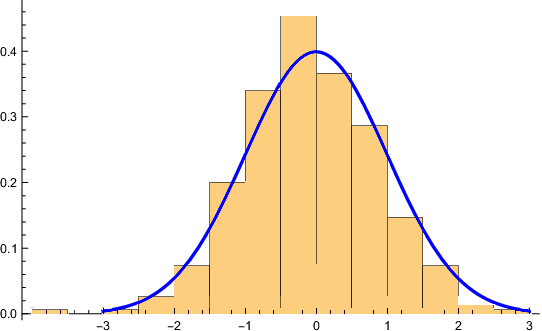}
\\
\end{center}
\begin{center}
\includegraphics[width=.3\textwidth]{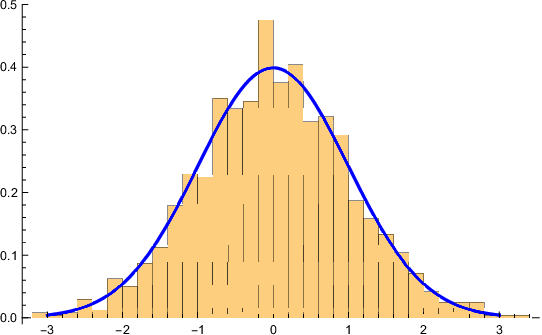}
\hspace{1em}
\includegraphics[width=.3\textwidth]{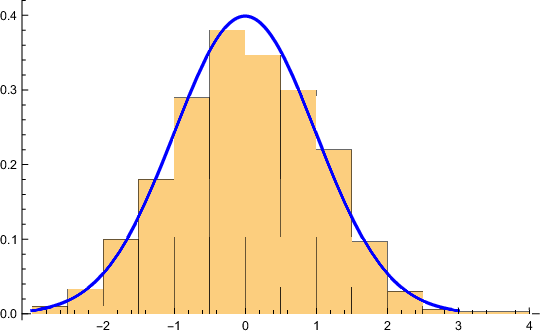}
\hspace{1em}
\includegraphics[width=.3\textwidth]{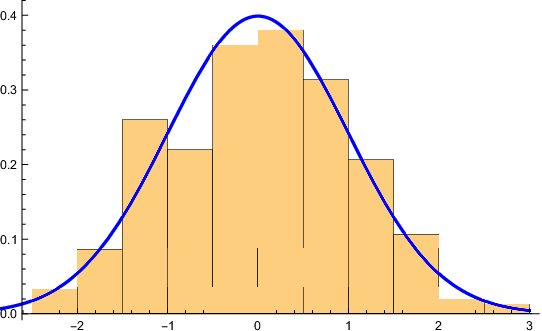}
\\
\end{center}
\caption{Distribution of z-scores for digits of the form $p-1$ (top row)
and $n^2-1$ (bottom row), corresponding to block sizes 
250,000 (left figure), 
500,000 (middle figure), 
and 1,000,000 (right figure), 
}
\label{fig:zscores}
\end{figure}

The histograms in Figure \ref{fig:zscores} suggest that the z-scores are
approximately normally distributed, as one would expect under the null
hypothesis that the continued fraction digits of $\pi$ of the form 
$p-1$ and $n^2-1$ behave like independent Bernoulli random variables,
with frequencies $q$ given by \eqref{eq:q-value-p1} and
\eqref{eq:q-value-s}, respectively.

For a more precise analysis we applied the  Kolmogorov-Smirnov and
Anderson-Darling goodness-of-fit tests to each set of z-scores,
comparing these scores to those drawn from a standard normal
distribution.  The resulting $p$-values are shown in Table
\ref{table:ks-ad-test}. None of these $p$-values is significant at the
$0.05$ level, thus providing further strong evidence for the 
``randomness'' of the continued fraction digits of $\pi$.  

\begin{table}[H]
\begin{center}
\renewcommand{\arraystretch}{1.5}
\begin{tabular}{|c|c|c|}
\hline
Block length & KS p-value&  AD p-value
\\
\hline
250,000
&0.4656
&0.1850
\\
500,000
&0.0779
&0.0525 
\\
1,000,000
&0.8915
&0.8821 
\\
\hline
\end{tabular}
\hspace{1em}
\begin{tabular}{|c|c|c|}
\hline
Block length & KS p-value&  AD p-value
\\
\hline
250,000
&0.6180 &0.9021
\\
500,000
&0.9077 & 0.7466 
\\
1,000,000
&0.0643 & 0.1076
\\
\hline
\end{tabular}
\caption{P-values from the Kolmogorov-Smirnov (KS)
and Anderson-Darling (AD) tests, applied to the z-scores
for digits of the form $p-1$ (left table) and $n^2-1$ (right
table).} 
\label{table:ks-ad-test}
\end{center}
\end{table}

\paragraph{Test III:}
Table \ref{table:longest-runs} shows the average longest run length,
with the average being taken over all blocks of length $N$, along with
the corresponding predicted values, given by the approximation 
\eqref{eq:longest-run-prediction} for the expected length of the longest
run of $1$s in a Bernoulli model with parameters $N$ and $q=1/\Phi^2$.

\begin{table}[H]
\begin{center}
\renewcommand{\arraystretch}{1.5}
\begin{tabular}{|c|c|c|c|}
\hline
Block length & Average longest run length& Prediction 
\\
\hline
250,000&
12.5267&
12.5142
\\
500,000&
13.2617&
13.2345
\\
1,000,000&
13.9567&
13.9547
\\
\hline
\end{tabular}
\caption{Average length of the longest run of $1$s in a block
of length $N$. The prediction is based on the asymptotic formula  
\eqref{eq:longest-run-prediction}.
}
\label{table:longest-runs}
\end{center}
\end{table}

\section{Concluding Remarks}
\label{sec:conclusion}

We conclude this paper by mentioning possible extensions and
generalizations of our results and some open questions suggested by
these results.

Proposition \ref{prop:subsets} shows that there is a close
connection between the subset probabilities $P(A)$ defined by 
\eqref{eq:PA-def} and identities for infinite products of the form 
$\prod_{b\in B}(1-1/b^2)$. Any set $B$ for which the latter product 
has a closed formula gives rise to a set $A$ for which $P(A)$ has a
closed formula of similar type.  Theorems \ref{thm:Pk} and
\ref{thm:S} exhibit two classes of sets for which such closed formulas
exist and are strikingly simple.  To keep the exposition simple, 
we decided to focus on these particular classes of sets, rather than
attempting to state and prove our results in their most general form.

Here we mention, without proof, one such generalization: 
Given an integer $k\ge 2$, let
$S_k$ be the set of ``shifted $k$th powers,'' i.e., $S_k=\{n^k-1:
n\in\NN, n>1\}$.  For $h=1,\dots,k-1$ let $\omega_h=e^{\pi i h/k}$.
Then
\begin{equation}
\label{eq:Sk}
P(A)=-\log_2\left|\frac{1}{2k\pi^{k-1}}
\prod_{h=1}^{k-1}\sin(\pi\omega_h)\right|.
\end{equation}
This result generalizes Theorem \ref{thm:S} and it can be proved using
the same approach, though the technical details are significantly more
involved. In fact, one can prove a slighly more general result for sets 
$S_{k,a}=\{an^k-1: n\in\NN\}$, where $a$ and $k$ are fixed 
positive integers.

Theorem \ref{thm:strings} on frequencies of strings of the form
$(a,\dots,a)$ can also be generalized in various directions. One such  
generalization is to strings of the form $\bfa=(A,\dots,A)$, consisting
of $k$ repeated blocks $A$, where $A$ is a given finite block of digits.
The resulting frequency $P(\bfa)$ can be expressed in terms of 
two-term linear recurrence sequences similar to the sequence $\Fa_n$ in
Theorem \ref{thm:strings}, but with coefficients and initial conditions
defined in terms of the block $A$.  In the case of 
a two-digit block $A=(a,b)$, repeated $k$ times, this frequency is given by
\begin{equation}
\label{eq:string-result-generalized}
P(\underbrace{(A,\dots,A)}_{k})=
P(\underbrace{(a,b,\dots,a,b)}_k)=
\left|\log_2\left(1+\frac{1} {\Fab_{k+2}\Gab_{k+2}}\right)\right|,
\end{equation}
where the sequences $\Fab_n$ and $\Gab_n$ are defined recursively by 
\begin{align*}
\label{eq:Fab-def}
&\Fab_1=b-1, \Fab_2=1,\quad \Fab_n=(ab+2)\Fab_{n-1} -\Fab_{n-2}
\quad(n\ge 3),
\\
&\Gab_1=a-1, \Gab_2=1,\quad \Gab_n=(ab+2)\Gab_{n-1} -\Gab_{n-2}
\quad(n\ge 3).
\end{align*}
We note that, from the general theory of linear recurrences, the
sequences $\Fab_n$ and $\Gab_n$ have explicit representations as linear
combinations of roots of the associated characteristic equation. 

In testing the predictions of our theorems, we focused on the continued
fraction expansion of the number $\pi$, in part because of the status of
$\pi$ as the most famous and most extensively studied mathematical
constant, but also because large sets of continued fraction digits of
$\pi$ can reliably be computed using off-the-shelf software such as
\emph{Mathematica}.  A possible direction for future research would be to
perform a systematic statistical analysis of the continued fraction
digits of particular classes of irrational numbers---for example,
algebraic numbers of degree greater than $2$.

Among open problems, the most glaring is a proof that the continued
fraction digits of $\pi$ and other ``natural'' irrational constants
are indeed ``random,'' in the sense that these numbers fall into the
class of ``almost all'' real numbers for which results such as the
Gauss-Kuzmin theorem and Theorems \ref{thm:Pk}--\ref{thm:strings} hold.
But this seems to be an intractable problem under currently available
methods.

\subsection*{Acknowledgements} This work originated with an undergraduate
research project carried out at the \emph{Illinois Geometry Lab}
(IGL) at the University of Illinois; we thank the IGL for providing this
opportunity.  We are also grateful to the referee for a careful reading
of this paper and many helpful suggestions. 





\end{document}